\documentclass[nointlimits,11pt,oneside]{amsart}
\usepackage{amssymb,cases,enumitem, verbatim}
\usepackage{color}
\usepackage{hyperref}
\hypersetup{
	    colorlinks=true,
	    linkcolor=blue,
	    citecolor=blue,
	    filecolor=green,
	    urlcolor=cyan,
	    bookmarks=true
	}
	
\makeatletter
\renewcommand*{\eqref}[1]{%
  \hyperref[{#1}]{\textup{\tagform@{\ref*{#1}}}}%
}
\makeatother

\setenumerate{label=(\roman*)}

\usepackage[%
	a4paper,
	total={16cm,23cm},
	left=1cm, top=3cm,
	marginparsep=2pt]
{geometry}

\theoremstyle{plain}
\newtheorem{theorem}{Theorem}[section]
\newtheorem{lemma}[theorem]{Lemma}

\newtheorem{proposition}[theorem]{Proposition}
\theoremstyle{definition}

\newtheorem{definition}[theorem]{Definition}
\newtheorem{example}[theorem]{Example}

\newtheorem*{notation}{Notation}
\numberwithin{equation}{section}

\def\L1loc{L^1_{\text{loc}}}

\begin{document}
This is the pre-peer reviewed version of the following article: Pe{\v s}a D. Reduction principle for a certain class of kernel-type operators. \emph{Mathematische Nachrichten}. 2020;1–13., which has been published in final form at  \url{https://doi.org/10.1002/mana.201800510}. This article may be used for non-commercial purposes in accordance with Wiley Terms and Conditions for Use of Self-Archived Versions.

\title{Reduction principle for a certain class of kernel-type operators}
\author{Dalimil Pe{\v s}a}

\address{Dalimil Pe{\v s}a, Department of Mathematical Analysis, Faculty of Mathematics and
Physics, Charles University, Sokolovsk\'a~83,
186~75 Praha~8, Czech Republic}
\email{pesa@karlin.mff.cuni.cz}
\urladdr{0000-0001-6638-0913}


\thanks{This research was supported by the grant SFG205 of Faculty of Mathematics and Physics, Charles University, Prague.}

\begin{abstract}
The classical Hardy--Littlewood inequality asserts that the integral of a product of two functions is always majorized by that of their non-increasing rearrangements. One of the pivotal applications of this result is the fact that the boundedness of an integral operator which integrates over some right neighbourhood of zero is equivalent to the boundedness of the same operator on the cone of positive non-increasing functions. It is well known that an analogous inequality for integration away from zero is not true. However, as we show in this paper, the equivalence of the restricted inequality for the non-restricted one is still true for certain class of kernel-type operators, regardless of the measure of the integration domain.
\end{abstract}

\date{\today}

\maketitle

\makeatletter
   \providecommand\@dotsep{2}
\makeatother

\section{Introduction}

The classical Hardy--Littlewood inequality asserts that for every pair of functions $f,g$ defined on a $\sigma$-finite measure space $(R,\mu)$, one always has
\[
\int_R|f(x)g(x)| \: d\mu(x) \leq \int_0\sp{\infty}f\sp*(t)g\sp*(t) \: dt,
\]
where $f\sp*,g\sp*$ denote the non-increasing rearrangements of $f,g$, respectively. It has applications all over the place, in particular in theory of Banach function spaces or in interpolation theory. An important particular case is the estimate
\[
\int_0\sp t |f(s)| \: ds \leq\int_0\sp t f\sp*(s) \: ds,
\]
which is valid for any measurable function $f$ on $(0,\infty)$. An interesting and useful consequence of this fact is that, given a nonnegative measurable function (weight) $w$ on $(0,\infty)$, the following two statements are equivalent:

\begin{enumerate}

\item there exists a~positive constant $C_1$ such that for every nonnegative and nonincreasing function $f$ the weighted Hardy-type inequality
\begin{equation*}\label{E:hardy}
\left\|w(t)\int_0\sp t f(s) \: ds \right\|_{Y(0,\infty)}
\leq C_1
\left\|f\right\|_{X(0,\infty)}
\end{equation*}
holds,

\item there exists a~positive constant $C_2$ such that for every nonnegative measurable function $f$ one has
\begin{equation*}\label{E:hardy2}
\left\|w(t)\int_0\sp tf(s) \: ds\right\|_{Y(0,\infty)}
\leq C_2
\left\|f\right\|_{X(0,\infty)}.
\end{equation*}

\end{enumerate}
Such an equivalence is often called a \textit{reduction principle} and it comes very handy in the research of mapping properties of operators and embeddings.

The crucial point in the Hardy--Littlewood inequality is that the integration takes place near zero, that is, over the interval $(0,t)$. If the integration interval is bounded away from zero (typically when integrating over $(t,\infty)$ rather than over $(0,t)$), a~statement analogous to the Hardy--Littlewood inequality is no longer true. Nevertheless, in many situations, reduction principles for operators involving integration away from zero are desirable. Pivotal examples are provided by the study of Sobolev embeddings, trace embeddings, or boundedness of important integral operators such as the Riesz potential, various modifications of the Hardy--Littlewood maximal operator, the Laplace transform, singular integrals, etc.

In connection with investigation of the relationship of isoperimetric profile of a~domain in a~Euclidean space to higher-order Sobolev embeddings it was shown in~\cite{CianchiPick15} that despite the lack of the appropriate Hardy--Littlewood inequality, there is still some chance for obtaining a sensible reduction principle, at least for operators of a~certain specific form and for weights that satisfy some monotonicity conditions. This result is quite deep, even surprising, and its proof is based on a combination of fine methods from real analysis with properties of the so-called \textit{down-dual} functionals known from the function space theory.
A principal restriction of the scope of applications of this result is however its restriction to finite intervals. For this reason it cannot be used for example when an action of potential operators or fractional maximal operators is investigated on function spaces built over the entire Euclidean space, which often arises in practical applications. In this paper, we fill in this gap and extend the result of~\cite{CianchiPick15} to the cases when integration takes part over an infinite measure space. Needless to say that this extension is far from being just some dull generalization. Indeed, a~new technique had to be developed in order to get it, although, naturally, the known results and methods have been exploited, too.

Let us now formulate our main result.

\begin{theorem}\label{MainResultLite}
Let $I:(0,\infty) \rightarrow (0,\infty)$ be a non-decreasing and let  $\lVert \cdot \rVert_X$ and  $\lVert \cdot \rVert_Y$ be rearrangement invariant Banach function norms on $M((0, \infty), \lambda)$. Then the following statements are equivalent:
\begin{enumerate}
\item There exists a~constant $C \in \mathbb{R}$ such that
\begin{equation}
\bigg \lVert \int_t^{\infty} \frac{f(s)}{I(s)}  \:  ds \bigg \rVert_Y \leq C \lVert f \rVert _X \label{TheoremML1}
\end{equation}
for all non-negative $f \in X$.
\item There exists a~constant $C' \in \mathbb{R}$ such that
\begin{equation}
\bigg \lVert \int_t^{\infty} \frac{f(s)}{I(s)}  \:  ds \bigg \rVert_Y \leq C' \lVert f \rVert _X \label{TheoremML2}
\end{equation}
for all non-increasing non-negative $f \in X$.

Furthermore, if \eqref{TheoremML2} holds, then \eqref{TheoremML1} holds with $C = 4C'$.
\end{enumerate}
\end{theorem}

We will in fact prove a stronger version of Theorem~\ref{MainResultLite}, but in order to formulate this result, that is, Theorem~\ref{MainResult}, some preliminary work is needed, namely one needs Definition~\ref{DefED} which is rather complicated, and for this reason we present here only a~simpler version of the result.

The paper is structured as follows. In the next section we collect all the background material and quote all the known facts which we are going to use in the sequel. In the final section we prove the main result.

\section{Preliminaries}

From now on, we will denote by $(R, \mu)$, and occasionally $(S, \nu)$,  some arbitrary sigma-finite measure space. When $E \subseteq R$, we will denote its characteristic function by $\chi_E$. The set of all extended complex-valued  $\mu$-measurable functions defined on $R$ will be denoted by $M(R, \mu)$, its subsets of all non-negative functions \footnote{That is, functions whose values are non-negative real numbers.} and functions finite $\mu$-almost everywhere on $R$ will be denoted by $M_+(R, \mu)$ and $M_0(R, \mu)$ respectively. As usual, we identify functions that are equal $\mu$-almost everywhere. For brevity, we will usually abbreviate $\mu$-almost everywhere to $\mu$-a.e.\ and simply write $M$, $M_+$ and $M_0$, instead of $M(R, \mu)$, $M_+(R, \mu)$ and $M_0(R, \mu)$ respectively, whenever there is no risk of confusion.

After preliminaries, we will restrict ourselves to the case $R = (0, \infty)$ and so we will denote the $1$-dimensional Lebesgue measure by $\lambda$.

\subsection{Non-increasing rearrangement} \label{SSNR}

In this section, we define the non-increasing rearrangement of a function and some related terms. We proceed in accordance with \cite[Chapter~2]{BennettSharpley88}. We first define the distribution function.

\begin{definition}
	The distribution function $\mu_f$ of a function $f \in M$ is defined for $s \in [0, \infty)$ by
	\begin{equation*}
		\mu_f(s) = \mu(\{ t \in R \vert f(t) > s \}).
	\end{equation*}
\end{definition}

We now define the non-increasing rearrangement as the generalised inverse of the distribution function.

\begin{definition} \label{DNIR}
	The non-increasing rearrangement $f^*$ of function $f \in M$ is defined for $t \in [0, \infty)$ by
	\begin{equation*}
		f^*(t) = \inf \{ s \in [0, \infty) \vert \mu_f(s) \leq t \}.
	\end{equation*}
\end{definition}

Some basic properties of distribution function and non-increasing rearrangement, with proofs, can be found in \cite[Chapter~2,~Proposition~1.3]{BennettSharpley88} and \cite[Chapter~2,~Proposition~1.7]{BennettSharpley88}. We now define what it means for functions to be equimeasurable.

\begin{definition}
	We say that two functions $f \in M(R, \mu)$, $g \in M(S, \nu)$ are equimeasurable if $\mu_f = \mu_g$.
\end{definition}

It is an easy exercise to prove that $f, g \in M$ are euqimeasurable if and only if also $f^* = g^*$.

A very important classical result is the Hardy-Littlewood inequality which we list below. For details, see for example \cite[Chapter~2,~Theorem~2.2]{BennettSharpley88}.

\begin{theorem} \label{THLI}
	It holds for all $f, g \in M$ that
	\begin{equation*}
		\int_R \lvert f \cdot g \rvert \: d\mu \leq \int_0^{\infty} f^* g^* \: d\lambda.
	\end{equation*}
\end{theorem}

As an immediate consequence, we get that, for all $f, g \in M$,
\begin{equation*}
	\sup_{\tilde{g} \in M, \tilde{g}^* = g^*} \int_R \lvert f \cdot \tilde{g} \rvert \: d\mu \leq \int_0^{\infty} f^* g^* \: d\lambda.
\end{equation*}
This leads to the definition of resonant measure spaces.\footnote{There is also a stronger version of resonance, which we omit since it will not be used. For details, see \cite[Chapter~2,~Definition~2.3]{BennettSharpley88}.}

\begin{definition}
	A sigma-finite measure space $(R, \mu)$ is said to be resonant if it holds for all $f, g \in M(R, \mu)$ that
	\begin{equation*}
		\sup_{\tilde{q} \in M(R, \mu), \tilde{q}^* = g^*} \int_R \lvert f \cdot \tilde{g} \rvert \: d\mu = \int_0^{\infty} f^* g^* \: d\lambda.
	\end{equation*}
\end{definition}

Characterization of resonant measure spaces can be found in \cite[Chapter~2,~Theorem~2.7]{BennettSharpley88}. For our purpose, a simple sufficient condition is enough.

\begin{theorem} \label{CharRes}
If the measure $\mu$ is non-atomic, then the measure space $(R, \mu)$ is resonant.
\end{theorem}

\subsection{Rearrangement invariant Banach function norms} \label{SSFN}

The following two definitions are adapted from \cite[Chapter~1,~Definition~1.1]{BennettSharpley88} and \cite[Chapter~2,~Definition~4.1]{BennettSharpley88} respectively.

 \begin{definition}\label{DefBFN}
	Let $\lVert \cdot \rVert : M_+ \to [0, \infty]$ be some non-negative functional on $M_+$. We then say that $\lVert \cdot \rVert$ is a Banach function norm if it satisfies the following conditions:

	\begin{enumerate}[label=\textup{(P\arabic*)}, series=P]
		\item \label{P1} $\lVert \cdot \rVert$ is a norm, i.e.
		\begin{enumerate}
			\item it is positively homogeneous, i.e.\ $\forall a \in \mathbb{C} \forall f \in M_+ \: : \: \lVert a \cdot f \rVert = \lvert a \rvert \lVert f \rVert$,
			\item it satisfies $\lVert f \rVert = 0 \Leftrightarrow f = 0$  $\mu$-a.e.,
			\item it is subadditive, i.e.\ $\forall f,g \in M_+ \: : \: \lVert f+g \rVert \leq \lVert f \rVert + \lVert g \rVert$.
		\end{enumerate}
		\item \label{P2} $\lVert \cdot \rVert$ has the lattice property, i.e.\ if some $f, g \in M_+$ satisfy $f \leq g$ $\mu$-a.e., then also $\lVert f \rVert \leq \lVert g \rVert$.
		\item \label{P3} $\lVert \cdot \rVert$ has the Fatou property, i.e.\ if  some $f_n, f \in M_+$ satisfy $f_n \uparrow f$ $\mu$-a.e., then also $\lVert f_n \rVert \uparrow \lVert f \rVert $.
		\item \label{P4} $\lVert \chi_E \rVert < \infty$ for all $E \subseteq R$ satisfying $\mu(E) < \infty$.
		\item \label{P5} For every $E \subseteq R$ satisfying $\mu(E) < \infty$ there exists some finite constant $C_E$, dependent only on $E$, such that for all $f \in M_+$ the inequality $ \int_E f \: d\mu \leq C_E \lVert f \rVert $ holds.
	\end{enumerate}
\end{definition}

There is one class of Banach function norms which will be of special interest for us, namely the rearrangement invariant Banach function norms defined bellow.

\begin{definition}
	We say that a Banach function norm $\lVert \cdot \rVert$ is rearrangement invariant, abbreviated r.i., if  it satisfies the following additional condition:
	\begin{enumerate}[resume*=P]
		\item \label{P6} If two functions $f,g \in M_+$ are equimeasurable, then $\lVert f \rVert = \lVert g \rVert$.
	\end{enumerate}
\end{definition}

While in it was convenient to define these terms only for functionals on $M_+$, their domains can be naturally expanded to whole $M$ by taking first the absolute value of given function. So we may say that $\lVert \cdot \rVert$ is, for example, Banach function norm on $M$ and mean by it that it is a functional on $M_+$ satisfying the definition above whose domain was expanded in this way. We shall do this implicitly from now on, without further reference. We may also mention the properties listed in the definitions above when talking about functionals defined on whole $M$. If we do so, we always mean that those functionals have said properties when restricted on $M_+$.

Now, having expanded the domain of $\lVert \cdot \rVert$ on whole $M$, we may define (r.i.) Banach function spaces

\begin{definition} \label{DefBFS}
	Let $\lVert \cdot \rVert_X$ be Banach function norm on $M$. Then the set
	\begin{equation*}
		X = \{ f \in M; \lVert f \rVert_X < \infty \}
	\end{equation*}
	equipped with the norm $\lVert \cdot \rVert$ will be called a Banach function space. Further, if $\lVert \cdot \rVert$ is rearrangement invariant, we shall say that $X$ is a rearrangement invariant Banach function space.
\end{definition}

Basic properties of Banach funtion norms and Banach function spaces can be found in \cite[Chapter~1,~Section~1]{BennettSharpley88}.

Important concept in the theory of Banach function spaces is the concept of associate space.

\begin{definition} \label{DefAss}
Let $\lVert \cdot \rVert_X$ be Banach function norm on $M$ and $X$ the corresponding Banach function space. Then the functional $\lVert \cdot \rVert_{X'}$ defined for every $f \in M$ by
\begin{equation*}
\lVert f \rVert_{X'} = \sup_{\lVert g \rVert_X \leq 1} \int_R \lvert fg \rvert \: d\mu
\end{equation*}
will be called the associate norm of $\lVert \cdot \rVert_X$ and and the space $X'$, defined as in Definition~\ref{DefBFS}, the associate space of $X$.
\end{definition}

Properties of associate spaces can be found in \cite[Chapter~1,~Section~2]{BennettSharpley88}. We now provide two important equalities, for details see \cite[Chapter~2,~Proposition~4.2]{BennettSharpley88}.

\begin{proposition} \label{PropAss}
Let $\lVert \cdot \rVert_X$ be a rearrangement invariant Banach function norm on $M(R, \mu)$ where $(R, \mu)$ is a resonant measure space. Then $\lVert \cdot \rVert_{X'}$ is also rearrangement invariant and we have the equalities
\begin{align*}
\lVert f \rVert_{X'} &= \sup_{\lVert g \rVert_X \leq 1} \int_0^{\infty} f^*g^* \: d\lambda, \\
\lVert g \rVert_ X &= \sup_{\lVert f \rVert_{X'} \leq 1} \int_0^{\infty} f^*g^* \: d\lambda.
\end{align*}
\end{proposition}

Another concept used in the paper is down-associate norm. To define it we need to restrict ourselves to the case $(R, \mu) = ((0, \infty), \lambda)$.

\begin{definition} \label{DefDownAss}
Let $\lVert \cdot \rVert_X$ be a rearrangement invariant Banach function norm on $M((0, \infty), \lambda)$ and $X$ the corresponding Banach function space. Then the functional $\lVert \cdot \rVert_{X'_d}$ defined for every $f \in M$ by
\begin{equation*}
\lVert f \rVert_{X'_d} = \sup_{\lVert g \rVert_X \leq 1} \int_0^{\infty} \lvert fg^* \rvert \: d\lambda
\end{equation*}
will be called the down-associate norm of $\lVert \cdot \rVert_X$ and and the space $X'_d$, defined as in Definition~\ref{DefBFS}, the down-associate space of $X$.
\end{definition}

It is fairly easy to check that the down associate norm is indeed a norm and in fact satisfies conditions~\ref{P1}--\ref{P4} from the Definition~\ref{DefBFS} of the Banach function spaces. Furthermore, it is obvious that $\lVert f \rVert_{X'_d} \leq \lVert f \rVert_{X'}$  for any $f \in M((0, \infty), \lambda)$, since the supremum in the definition above is essentially the supremum from Definition~\ref{DefAss} of the associate space but taken only over non-increasing functions $g$, and therefore we always have the embedding $X' \hookrightarrow X'_d$.

To show a concrete example one can use the characterisation of embeddings of classical Lorentz spaces proved in \cite[Remark~(i),~p.~148]{Sawyer90} and summarized in \cite[Theorem~3.1]{CarroPick00} to obtain
\begin{align*}
\lVert f \rVert_{(L^1)'_d} &= \sup_{t \in (0, \infty)} \frac{1}{t} \int_0^{t} f(s) \: ds, \\
 \lVert f \rVert_{(L^p)'_d} &\approx \left ( \int_0^{\infty} \left ( \frac{1}{t} \int_0^{t} f(s) \:ds \right )^{\frac{1}{p-1}}f(t) \: dt \right )^{\frac{p-1}{p}}, & \text{where } p \in (1, \infty),
\end{align*}
for any $f \in M_+$, where the symbol $\approx$ means that the ratio of left and right hand sides is bounded between two positive constants depending only on $p$. In the latter case,  \cite[Theorem~1]{Sawyer90} also provides an equivalent and perhaps nicer expression
\begin{align*}
 \lVert f \rVert_{(L^p)'_d} \approx \left \lVert \int_t^{\infty} \frac{f(s)}{s} \: ds \right \rVert_{L^{\frac{p}{p-1} }} & \text{where } p \in (1, \infty).
\end{align*}
\subsection{Operators}

To conclude this section we list two basic definitions concerning operators.

\begin{definition}
Let $T:M_{+} \rightarrow M_{+}$  be an sublinear operator. Given rearrangement invariant Banach function spaces $X$ and $Y$, we say that $T$ is bounded from $X$ to $Y$, and write
\begin{equation*}
T:X \rightarrow Y
\end{equation*}
if the quantity
\begin{equation*}
\lVert  T \rVert = \sup \lbrace \lVert Tf \rVert _{Y} \vert f \in X \cap M_{+}, \lVert f \rVert_X \leq 1 \rbrace
\end{equation*}
is finite. $\lVert T \rVert$ is then said to be the norm of $T$.

\end{definition}

\begin{definition}
Let $T$ and $T'$ be two operators from $M_{+}$ into $M_{+}$. We say that $T$ and $T'$ are mutually associate, if

\begin{equation*}
\int_R Tf \cdot g  \: d\mu = \int_R f(s) \cdot T'g  \: d\mu
\end{equation*}
for all $f, g \in M_{+}$.
\end{definition}

\section{The fall of the star}

On the following pages, we adapt the methods used by Cianchi, Pick and Slav{\' i}kov{\' a} in \cite[Section~9]{CianchiPick15} to extend some of their result to the case when the underlying measure space is of infinite measure.

As foreshadowed before, we restrict ourselves to the case $R = (0, \infty)$ and $\mu = \lambda$.

\subsection{Auxiliary statements}

\begin{definition} \label{BDef}
Let $I:(0, \infty) \rightarrow (0, \infty)$ be a non-decreasing function. We define the operators $R_{I}$ and $H_{I}$ from $M_{+}$ into  $M_{+}$ by
\begin{align*}
R_{I}f(t) &= \frac{1}{I(t)} \int_{0}^{t} f(s)  \: ds  & \text{for} \  t \in (0, \infty), \\
H_{I}f(t) &= \int_{t}^{\infty}  \frac{f(s)}{I(t)}  \: ds & \text{for} \  t \in (0, \infty),
\end{align*}
where $f \in M_{+}$. Furthermore, for $m \in \mathbb{N}$, we set
\begin{align*}
R_{I}^m &= \underbrace{R_I \circ R_I \circ \dotsb \circ R_I}_{m \text{-times}} & \text{and}   & &  H_{I}^m = \underbrace{H_I \circ H_I \circ \dotsb \circ H_I}_{m \text{-times}}.
\end{align*}
We also formally set $R_I^0$ and $H_I^0$ to be the identity operator on $M_{+}$.

\end{definition}

Some basic properties of these operators are listed in the following proposition. The proof is easy and therefore omitted.

\begin{proposition} \label{BProp} \ \\
\begin{enumerate}
\item \label{BProp1} The operators $R_I^m$ and $H_I^m$ are mutually associate for all $m \in \mathbb{N} \cup \{ 0 \}$.
\item  For every $f \in M_{+}$, every $m \in \mathbb{N}$ and every $t \in (0,\infty)$ holds, that
\begin{equation}
 R_I^mf(t) \leq R_I^mf^*(t). \label{BProp2}
\end{equation}
\item  For every $f \in M_{+}$, every $m \in \mathbb{N}$ and every $t \in (0,\infty)$ holds, that
\begin{align}
R_I^mf(t) &= \frac{1}{(m-1)!}\frac{1}{I(t)} \int_0^{t} f(s) \bigg( \int_s^t \frac{1}{I(r)} \: dr\bigg)^{m-1}  \: ds \label{BProp3},\\
H_I^mf(t) &= \frac{1}{(m-1)!} \int_t^{\infty} \frac{f(s)}{I(s)} \bigg( \int_t^s \frac{1}{I(r)} \: dr\bigg)^{m-1}  \:  ds \label{BProp4}.
\end{align}
\item \label{BProp5} The operators $R_I^m$ and $H_I^m$ are monotone for every $m \in \mathbb{N}$, in the sense that if $f \leq g$ a.e.\ then also $R_I^mf \leq R_I^mg $ and $H_I^mf \leq H_I^mg$.
\end{enumerate}
\end{proposition}

In the next lemma we state the critical property of the operator $R^m_I$ that is at the heart of our proof. The proof is identical to the one in \cite[Lemma~9.1]{CianchiPick15}.

\begin{lemma} \label{LemmaR}
 It holds for all $f \in M_+$, all $m \in \mathbb{N} \cup \{ 0 \}$ and all $t \in (0, \infty)$ that
\begin{align}
R_I^mf^*(t) &\leq 2^{m} R_I^mf^*(s) & \text{if} \  \frac{t}{2} \leq s \leq t. \label{LemmaR1}
\end{align}
Consequently, for every $f \in M_+$
\begin{align}
(d - c)R_I^mf^*(d) &\leq 2^{m+1} \int_c^d R_I^mf^*(t) \:dt & \text{for any} \  0 \leq c \leq d < \infty. \label{LemmaR2}
\end{align}
\end{lemma}

To prove the main result, we will need to utilize one additional  operator.

\begin{definition} \label{DefG}
Let $I: (0, \infty) \rightarrow (0, \infty)$ be a non-decreasing function, let $m \in \mathbb{N}$ and let $R_I^m$ be the operator defined in Definition~\ref{BDef}, then we define operator $G_I^m$ for every $f \in M_+$ by:
\begin{align*}
G_I^mf(t) &= \sup_{s \geq t} R_I^mf^*(s) & \text{for} \  t \in (0, \infty).
\end{align*}
If $m=1$ then we simply denote $G_I^1$ by $G_I$.
\end{definition}

It follows immediately from Definition~\ref{DefG} that $G_I^mf \geq R_I^mf^*$ for all $f \in M_+$ and that $G_I^m$ is non-increasing, which implies that $G_I^mf = (G_I^mf)^* \geq (R_I^mf^*)^* $.

Following lemma will tell us that the quantity $\lVert G_I^mf \rVert$, where $\lVert \cdot \rVert$ represents any Banach function norm on $M_+$, does not change when we replace $I$ with its left-continuous representative. The proof is again omitted since its differences from the proof in \cite[Lemma~9.2]{CianchiPick15} are only cosmetic.

\begin{lemma} \label{LemmaG} Let $m \in \mathbb{N}$, let $I:(0, \infty) \rightarrow (0, \infty)$ be a non-decreasing function, and let  $I_0:(0, \infty) \rightarrow (0, \infty)$ be the non-decreasing left-continuous function that coincides with $I$ a.e. on $(0,\infty)$. Then for every $f \in M_+$
\begin{equation*}
G_I^mf(t)=G_{I_0}^mf(t)
\end{equation*}
for all $t \in ( (0,\infty) \setminus M)$, where $M$ is some at most countable subset of $(0,\infty)$, i.e. the equality holds $\lambda$-a.e.\ on $(0,\infty)$, and consequently, given any Banach function norm $\lVert \cdot \rVert$ on $M_+$, $\lVert G_I^mf \rVert=\lVert G_{I_0}^mf \rVert$ for every $f \in M_+$.
\end{lemma}

Before we formulate the last necessary lemma we will introduce some new notation and one new term.

\begin{notation}
For the sake of brevity, we will from now on use the notation
\begin{equation*}
\Phi_I^m(t, s) =  \frac{\left ( \int_s^t \frac{1}{I(\tau)} \: d\tau \right )^{m-1}}{I(t)}.
\end{equation*}
\end{notation}

\begin{definition}\label{DefED}
Let $I:(0, \infty) \rightarrow (0, \infty)$ be a left-continuous non-decreasing function and fix some $m \in \mathbb{N}$. We say, that the $\Phi_I^m(t, s)$ is essentially decreasing in $t$ if the following condition holds:
\begin{equation}
\begin{gathered}
\forall s_0 \in (0, \infty) \exists t_0 \in (0, \infty) \forall t \in (t_0, \infty) \exists r_t \in (0, \infty) \forall r \in (r_t, \infty) \forall s \in (0, s_0) :\\
 \Phi_I^m(t, s) \geq \Phi_I^m(r, s). \label{DefED1}
\end{gathered}
\end{equation}
\end{definition}

We recognize that the above condition is rather complicated, but that is necessary in order for it to be as weak as possible. A much simpler but stronger conditions are for example:
\begin{enumerate}
\item \label{DefEDi} For every $s_0 \in (0, \infty)$ there is some deleted neighbourhood of infinity where the function $\Phi_I^m(t, s)$ is non-increasing with respect to $t$ for all choices of $s \in (0, s_0)$.
\item For every $s_0 \in (0, \infty)$ the function $\Phi_I^m(t, s)$ converges to $0$, as $t$ goes to infinity, uniformly for $s \in (0, s_0)$.
\end{enumerate}

Note that if $m=1$, then $\Phi_I^m(t, s)$ is simply $\frac{1}{I(t)}$ and thus essentially decreasing in $t$ for any non-decreasing left-continuous $I$. An example of functions $I$ that generate $\Phi_I^m(t, s)$ which are essentially decreasing in $t$ even for greater $m$ follows.

\begin{example}
If we put $I(t) = t^{\alpha}, \ \alpha \geq 1$, then for all $m \in \mathbb{N}$ the function $\Phi_I^m(t, s)$ is essentially decreasing in $t$ since it satisfies the condition~\ref{DefEDi}.
\end{example}

We now present the final lemma of this subsection. It is presented with proof, since it differs significantly from the one in \cite[Proposition~9.3]{CianchiPick15}. In fact, this is the part that needed the greatest modification and which motivates the Definition~\ref{DefED}.

\begin{lemma} \label{LemmaE}
	Let $m \in \mathbb{N}$, let $I:(0, \infty) \rightarrow (0, \infty)$ be a left-continuous non-decreasing function such that the function $\Phi_I^m(t, s)$ is essentially decreasing in $t$, and let $f \in M_+$ be function, for which the set $S_f = \{t \in (0, \infty)\vert f(t) > 0 \}$ has finite measure, i.e. there there exists some $s_f < \infty$ such that $\lambda (S_f) = s_f$. Then set E, defined by
	\begin{equation}
		E = \{t \in (0, \infty) \vert R_I^mf^*(t) < G_I^mf(t) \} \label{DefE}
	\end{equation}
	is an open subset of $(0, \infty)$ such that there exist at most countable collection of disjoint bounded open intervals $ \{ (c_k, d_k) \vert k \in I \subseteq \mathbb{N} \} $ in $(0, \infty)$ satisfying
	\begin{align}
		E &= \bigcup_{k \in I} (c_k, d_k), \label{PartsE} \\
		G_I^mf(t) &= R_I^mf^*(t) & \text{for} \ t \in ((0, \infty) \setminus E), \label{AssE1} \\
		G_I^mf(t) &= R_I^mf^*(d_k) & \text{if} \ t \in (c_k, d_k) \ \text{for any} \ k \in I. \label{AssE2}
	\end{align}
\end{lemma}

\begin{proof}
At first we shall prove three crucial properties of the function $R_I^mf^*$, namely:
\begin{enumerate}
\item \label{LemmaE1} If $R_I^mf^*(t) = \infty$ for any $t \in (0,\infty)$, then $R_I^mf^*(t) = \infty$ for all $t \in (0,\infty)$.
\item \label{LemmaE2} There is a $t_0 \geq s_f$ such that for every $t \geq t_0$ there is some $r_t \geq t$ satisfying that for all $r \geq r_t$ the inequality $R_I^mf^*(t) \geq R_I^mf^*(r)$ holds.
\item \label{LemmaE3} $R_I^mf^*$ is upper semi-continuous and hence attains its supremum over every closed interval.
\end{enumerate}

To prove \ref{LemmaE1}, note that, thanks to $f^*$ being non-increasing, the quantity $\int_a^t f^*(s) \big( \int_s^t \frac{1}{I(r)} \: dr\big)^{m-1} \: ds$ is finite for any $a>0$ and any $t\in [a, \infty)$, since
\begin{equation*}
\begin{split}
\int_a^t f^*(s) \bigg( \int_s^t \frac{1}{I(r)} \: dr\bigg)^{m-1} \: ds &\leq \int_a^t f^*(a) \bigg( \int_s^t \frac{1}{I(r)} \: dr\bigg)^{m-1} \: ds \\
 &= f^*(a) \int_a^t \bigg( \int_s^t \frac{1}{I(r)} \: dr\bigg)^{m-1} \: ds
\end{split}
\end{equation*}
which is finite, because $\big( \int_s^t \frac{1}{I(r)} \: dr\big)^{m-1}$ is continuous for $s \in [a, t]$. Hence, if $\int_0^t f^*(s) \big( \int_s^t \frac{1}{I(r)} \: dr\big)^{m-1} \: ds = \infty$ for any $t \in (0, \infty)$, it is also infinite for all $t \in (0, \infty)$. Conclusion~\ref{LemmaE1} follows, because, by \eqref{BProp3}, $R_I^mf^*(t)$ is in any $t \in (0, \infty)$ only this quantity multiplied by finite number, and therefore it follows the same rule.

As for~\ref{LemmaE2}, observe that since $f=0$ everywhere outside of set $S_f$ (of measure $s_f$), we have, by Definition~\ref{DNIR}, that $f^*(t) = 0$ for all $t \in [s_f, \infty)$. Hence, we can express $R_I^mf^*(t)$ in any $t \in [s_f, \infty)$ by \eqref{BProp3} to get
\begin{equation*}
\begin{split}
R_I^mf^*(t) &= \frac{1}{(m-1)!} \frac{1}{I(t)} \int_0^t f^*(s) \bigg( \int_s^t \frac{1}{I(r)} \: dr\bigg)^{m-1} \: ds \\
  &= \frac{1}{(m-1)!} \int_0^{s_f} f^*(s) \bigg( \int_s^t \frac{1}{I(r)} \: dr\bigg)^{m-1}\frac{1}{I(t)} \: ds.
\end{split}
\end{equation*}
We now see that in order to get $R_I^mf^*(t) \geq R_I^mf^*(r)$ it suffices to have $\Phi_I^m(t, s) \geq \Phi_I^m(r, s)$ for all $s \in (0, s_f)$, so to obtain the required $t_0$ we need only to use the assumption that $\Phi_I^m(t, s)$ is essentially decreasing in $t$.

The function $R_I^mf^*$ is upper semi-continuous simply because $IR_I^mf^*$ is continuous and $\frac{1}{I}$ is upper semi-continuous, since $I$ is non-decreasing and left-continuous and thus lower semi-continuous. That $R_I^mf^*$ attains its supremum over every compact set, specially over every closed interval, is simple consequence.

As an immediate consequence of~\ref{LemmaE2}, it holds for every $t \in (0, \infty)$ that in order to be unbounded on $[t, \infty)$, $R_I^mf^*$ has to be unbounded on either $[t, r_t]$ or $[t, r_{t_0}]$, depending on whether $t \geq t_0$ or not. So whenever $G_I^mf(t) = \infty$, then $R_I^mf^*$ is unbounded on some closed interval and therefore we have by~\ref{LemmaE3} that there exists some point $s \in [t, \infty)$ such that $R_I^mf^*(s) = \infty$ and thus, by~\ref{LemmaE1}, both $R_I^mf^*$ and $G_I^mf$ are identically equal to $\infty$ and there is nothing to prove ($E = \emptyset$). We may therefore assume $G_I^mf(t) <\infty$ on $(0, \infty)$.

We now distinguish two cases. Firstly, suppose $t \geq t_0$. Then, by~\ref{LemmaE2}, we get $G_I^mf(t) = \sup_{s \in [t, \infty)} R_I^mf^*(s) = \sup_{s \in [t, r_t]} R_I^mf^*(s)$ which is attained by~\ref{LemmaE3}. On the other hand, if $t < t_0$, then
\begin{equation*}
\begin{split}
G_I^mf(t) &= \sup_{s \in [t, \infty)} R_I^mf^*(s) = \max \left \{ \sup_{s \in [t, t_0]} R_I^mf^*(s), \sup_{s \in [t_0, \infty)} R_I^mf^*(s) \right \} \\
&= \max \left \{ \sup_{s \in [t, t_0]} R_I^mf^*(s), \sup_{s \in [t_0, r_{t_0}]} R_I^mf^*(s) \right \} = \sup_{s \in [t, r_{t_0}]} R_I^mf^*(s)
\end{split}
\end{equation*}
which is again attained by~\ref{LemmaE3}. Thus in either case we get that for every $t \in (0, \infty)$ there is some $c_t \geq t$ such that $G_I^mf(t) = R_I^mf^*(c_t)$. We remark that it holds for such $c_t$ that $G_I^mf(c_t) = R_I^mf^*(c_t)$ since
\begin{equation*}
 R_I^mf^*(c_t) = G_I^mf(t) \geq G_I^mf(c_t) \geq R_I^mf^*(c_t).
\end{equation*}

Suppose now that $t \in E$. Then, by \eqref{DefE}, $R_I^mf^*(t) < G_I^mf(t)$ and thus, thanks to upper semi-continuity of $R_I^mf^*$, there exists some $ \delta >0$ such that
\begin{align}
R_I^mf^*(s) &< G_I^mf(t) & \text{if} \ s \in (t-\delta, t+\delta). \label{LemmaE4}
\end{align}
Obviously $c_t > t+\delta$ which for $s \in (t, t + \delta)$ implies
\begin{equation*}
G_I^mf(t) \geq G_I^mf(s) \geq R_I^mf^*(c_t) = G_I^mf(t).
\end{equation*}
For $s \in (t-\delta, t)$ we have trivially $G_I^mf(s) \geq G_I^mf(t)$, but the sharp inequality is impossible, since it would mean that there exists some $r \in [s, t)$ such that $R_I^mf^*(r) > G_I^mf(t)$, which contradicts \eqref{LemmaE4}. Hence, we have that $G_I^mf(s) = G_I^mf(t) > R_I^mf^*(s) $ for all $s \in (t-\delta, t+\delta)$ and E is therefore an open set.

Conclusion \eqref{PartsE} is simple, since open intervals form a base of open sets on $(0, \infty)$, so $E$ is union of some collection of open intervals, and, if we take instead of every interval of the original collection the maximal interval containing it which is still subset of $E$ and eliminate any duplicities, we have $E$ expressed as a collection of disjoint open intervals and any disjoint collection of open sets on separable space, specially $(0, \infty)$, is at most countable. That all such intervals are bounded follows from the fact, that for any $t \in E$ the point $c_t \notin E$. It remains only name the endpoints and index them by the elements of some $I \subseteq \mathbb{N}$ to get the expression \eqref{PartsE}.

The conclusion \eqref{AssE1} follows immediately from the definition of $E$. As for \eqref{AssE2}, because all $s \in (t, d_k)$ belong to $E$, we have that the supremum $G_I^mf(t) = \sup_{s \in [t, \infty)} R_I^mf^*(s)$ must be attained somewhere in $[d_k, \infty)$ and is therefore equal to $\sup_{s \in [d_k, \infty)} R_I^mf^*(s) = G_I^mf(d_k)$. Conclusion \eqref{AssE2} follows, because $d_k \notin E$ is ensured by our assumption of maximality of the interval $(c_k, d_k)$.
\end{proof}

\subsection{Main result}

We now move to prove the main result. Most of the work is done in the next Theorem. To prove it, we adapt the proof that can be found in \cite[Theorem~9.5]{CianchiPick15} and expand it to fit our needs.

\begin{theorem} \label{TheoremE}
Let $I:(0, \infty) \rightarrow (0, \infty)$ be a non-decreasing left-continuous function and let $\lVert \cdot \rVert_X$ be rearrangement invariant Banach function norm on $M_+$. Let $m \in \mathbb{N}$. Suppose, that $\Phi_{I}^m(t, s)$ is essentially decreasing in $t$. Then
\begin{equation}
\lVert R_I^mf^* \rVert_{X'_d} \leq \lVert R_I^mf^* \rVert_{X'} \leq \lVert G_I^mf \rVert_{X'} \leq 2^{m+1} \lVert R_I^mf^* \rVert_{X'_d} \label{NormEquiv}
\end{equation}
for every $f \in M_+$.
\end{theorem}

Note that the assumption that $I$ is left-continuous is without loss of generality, because expression \eqref{NormEquiv} is not affected by replacement of $I$ by its left-continuous representative. Indeed, the first and last quantities in \eqref{NormEquiv} will not be affected, since the latter can differ from the former only on countable subset of $(0, \infty)$, which means that also $R_I^mf^*$ will change only on countable subset of $R_I^mf^*$ , while $\lVert G_I^mf \rVert_{X'}$ will remain the same by Lemma~\ref{LemmaG}.

\begin{proof}
The inequality $\lVert R_I^mf^* \rVert_{X'} \leq \lVert G_I^mf \rVert_{X'}$ is trivial, since $R_I^mf^* \leq G_I^mf$, just as it is trivial that $\lVert R_I^mf^* \rVert_{X'_d} \leq \lVert R_I^mf^* \rVert_{X'}$. The remaining part, i.e.\ that
\begin{equation}
\lVert G_I^mf \rVert_{X'} \leq 2^{m+1}\lVert R_I^mf^* \rVert_{X'_d} \label{NormIneq}
\end{equation}
will be proven in two steps.

At first we prove \eqref{NormIneq} for those $f \in M_+$, for which the set $S_f = \{t \in [0, \infty) \vert f(t) > 0 \}$ has finite measure, i.e. there exists $s_f < \infty$ such that $\lambda (S_f) = s_f$. We may apply at such $f$ Lemma~\ref{LemmaE}, and thus if we define $E$ as in \eqref{DefE} we can define $\{ (c_k, d_k) \vert k \in I \subseteq \mathbb{N} \}$ as the at most countable collection of disjoint open intervals satisfying \eqref{PartsE}, \eqref{AssE1} and \eqref{AssE2}. We then define, for every $g \in X$ (which of course is the rearrangement invariant Banach function space defined by the norm $\lVert \cdot \rVert_X$), the operator $A:M_+ \rightarrow M_+$ by
\begin{equation}
A(g)=g^* \chi_{(0,\infty) \setminus E} + \sum_{k \in I} \bigg( \chi_{(c_k,d_k)} \frac{1}{d_k - c_k} \int_{c_k}^{d_k} g^*(s) \: ds \bigg). \label{DefA}
\end{equation}
$A(g)$ is obviously non-increasing. Moreover, it is also an averaging operator in the sense of \cite[Chapter~2,~Theorem~4.8]{BennettSharpley88} and thus, by the same theorem,
\begin{equation}
\lVert g \rVert_X \leq 1 \Rightarrow \lVert A(g) \rVert_X \leq 1. \label{IneqA}
\end{equation}
Therefore we may for such $g$ derive the following chain of equalities and inequalities:
\begin{align*}
\int_0^{\infty} g^*(t) G_I^mf(t) \: dt &= \int_{(0,\infty) \setminus E} g^*(t) R_I^mf^*(t) \: dt \\
 &+ \sum_{k \in I} \int_{c_k}^{d_k} g^*(t) R_I^mf^*(d_k) \: dt & \text{(by \eqref{AssE1} and \eqref{AssE2})} \\
 &= \int_{(0,\infty) \setminus E} g^*(t) R_I^mf^*(t) \: dt \\
 &+ \sum_{k \in I} \bigg( \frac{1}{c_k - d_k} \int_{d_k}^{c_k} g^*(t)\: dt \bigg) (d_k - c_k)  R_I^mf^*(d_k)\\
 &\leq \int_{(0,\infty) \setminus E} A(g)(t) R_I^mf^*(t) \: dt \\
 &+ 2^{m+1}\sum_{k \in I} \int_{d_k}^{c_k} A(g)(t)R_I^mf^*(t) \: dt  & \text{(by \eqref{DefA} and \eqref{LemmaR2})} \\
 &\leq 2^{m+1} \int_0^{\infty} A(g)(t)R_I^mf^*(t) \: dt \\
 &\leq 2^{m+1} \sup_{ \lVert h \rVert_X \leq 1} \int_0^{\infty} h^*(t) R_I^mf^*(t) \: dt & \text{(by \eqref{IneqA})} \\
 &= 2^{m+1} \lVert R_I^mf^* \rVert_{X'_d} & \text{(by the definition of $\lVert \cdot \rVert_{X'_d}$)}.
\end{align*}
By taking the supremum over unit ball in $X$, we get the desired inequality (thanks to $G_I^mf$ being non-increasing)
\begin{equation*}
\lVert G_I^mf \rVert_{X'} = \lVert G_I^mf \rVert_{X'_d} \leq 2^{m+1} \lVert R_I^mf^* \rVert_{X'_d},
\end{equation*}
i.e. we have proven the inequality \eqref{NormIneq} for all $f \in M_+$ satisfying the additional condition $\lambda( \{ S_f \}) = s_f$ for some $s_f < \infty$.

In order to prove \eqref{NormIneq} for all $f \in M_+$, fix some arbitrary $f$ and consider the sequence $ \{ f_n \}_{ n \in \mathbb{N}} $ of function in $M_+$ defined for every $n \in \mathbb{N}$ by
\begin{equation}
 f_n = f \chi_{(0,n)}. \label{DefFn}
\end{equation}
Then obviously one has for all $f_n$ that $\lambda( \{ S_f \}) = s_f$ for some $s_f \leq n < \infty$ and therefore, by the proof above, we have that for all $n \in \mathbb{N}$
\begin{equation}
\lVert G_I^mf_n \rVert_{X'} \leq 2^{m+1}\lVert R_I^mf^*_n \rVert_{X'_d}. \label{IneqFn}
\end{equation}
Because $\{f_n\}_{n \in \mathbb{N}}$ is obviously non-decreasing and converges pointwise to $f$, it follows by \cite[Chapter~2,~Theorem~1.7]{BennettSharpley88} that the sequence $\{f^*_n\}_{n \in \mathbb{N}}$ is also non-decreasing and  converges pointwise to $f^*$. Hence, $\{R_I^mf_n^*\}_{n \in \mathbb{N}}$ is non-decreasing too (by Proposition~\ref{BProp}, part~\ref{BProp5}), and satisfies, by the classical monotone convergence theorem and \eqref{BProp3}, following equality for every $t \in (0, \infty)$
\begin{equation}
\begin{split}
R_I^mf^*(t) &= \frac{1}{(m-1)!}\frac{1}{I(t)} \int_0^{t} f^*(s) \bigg( \int_s^t \frac{1}{I(r)} \: dr\bigg)^{m-1}  \: ds \\
 &= \frac{1}{(m-1)!}\frac{1}{I(t)} \int_0^{t} \lim_{n \rightarrow \infty}f_n^*(s) \bigg( \int_s^t \frac{1}{I(r)} \: dr\bigg)^{m-1}  \: ds \\
 &= \lim_{n \rightarrow \infty} \frac{1}{(m-1)!}\frac{1}{I(t)} \int_0^{t} f_n^*(s) \bigg( \int_s^t \frac{1}{I(r)} \: dr\bigg)^{m-1}  \: ds \\
 &= \lim_{n \rightarrow \infty} R_I^mf_n^*(t). \label{LimitR}
\end{split}
\end{equation}
Now, $\lVert \cdot \rVert_{X'_d}$ is not a Banach function norm, but it has the Fatou property~\ref{P3} which when combined with the above observations yields
\begin{equation}
\lVert R_I^mf_n^* \rVert_{X'_d} \uparrow \lVert R_I^mf^* \rVert_{X'_d}. \label{LimitNR}
\end{equation}
Furthermore, by the Definition~\ref{DefG} of $G_I^m$, we have
\begin{align*}
G_I^mf(t) &= \sup_{s \geq t} R_I^mf^*(t) = \sup_{s \geq t} R_I^m(\lim_{n \rightarrow \infty} f_n^*(t)) \\
 &= \sup_{s \geq t} \lim_{n \rightarrow \infty} R_I^mf_n^*(t)  & \text{(by \eqref{LimitR})}\\
 &= \lim_{n \rightarrow \infty} \sup_{s \geq t} R_I^mf_n^*(t)  & \text{(since $\{R_I^mf_n^*\}_{n \in \mathbb{N}}$ is non-decreasing)} \\
 &= \lim_{n \rightarrow \infty} G_I^mf_n(t).
\end{align*}
Since the sequence $\{G_I^mf_n\}_{n \in \mathbb{N}}$ is obviously non-decreasing we get by \cite[Chapter~1,~Theorem~1.5]{BennettSharpley88} that
\begin{equation}
\lVert G_I^mf_n \rVert_{X'} \uparrow \lVert G_I^mf \rVert_{X'}. \label{LimitNG}
\end{equation}
We may now prove the desired inequality \eqref{NormIneq}.
\begin{align*}
\lVert G_I^mf \rVert_{X'} &= \lim_{n \rightarrow \infty} \lVert G_I^mf_n \rVert_{X'} & \text{(by \eqref{LimitNG})} \\
 &\leq 2^{m+1} \lim_{n \rightarrow \infty} \lVert R_I^mf_n^* \rVert_{X'_d} & \text{(by \eqref{IneqFn})} \\
 &= 2^{m+1} \lVert R_I^mf^* \rVert_{X'_d}  & \text{(by \eqref{LimitNR})}.
\end{align*}
\end{proof}

It remains only to show that this indeed implies the result we desired. Albeit the proof is almost identical to the one in \cite[Theorem~5.1]{CianchiPick15}, we present it here for the sake of completeness.

\begin{theorem}\label{MainResult}
Let $m \in \mathbb{N}$, let $I:(0,\infty) \rightarrow (0,\infty)$ be a non-decreasing left-continuous function and let  $\lVert \cdot \rVert_X$ and  $\lVert \cdot \rVert_Y$ be rearrangement invariant Banach function norms on $M$. Suppose, that $\Phi_{I}^m(t, s)$ is essentially decreasing in $t$. Then the following statements are equivalent:
\begin{enumerate}
\item \label{Allf} There exists a~constant $C \in \mathbb{R}$ such that
\begin{equation}
\bigg \lVert \int_t^{\infty} \frac{f(s)}{I(s)} \bigg( \int_t^s \frac{1}{I(r)} \: dr\bigg)^{m-1}  \:  ds \bigg \rVert_Y \leq C \lVert f \rVert _X \label{TheoremM1}
\end{equation}
for all non-negative $f \in X$.
\item \label{Nif} There exists a~constant $C' \in \mathbb{R}$ such that
\begin{equation}
\bigg \lVert \int_t^{\infty} \frac{f(s)}{I(s)} \bigg( \int_t^s \frac{1}{I(r)} \: dr\bigg)^{m-1}  \:  ds \bigg \rVert_Y \leq C' \lVert f \rVert _X \label{TheoremM2}
\end{equation}
for all non-increasing non-negative $f \in X$.

Furthermore, if \eqref{TheoremM2} holds, then \eqref{TheoremM1} holds with $C = 2^{m+1}C'$.
\end{enumerate}
\end{theorem}

Note that, as in Theorem~\ref{TheoremE}, the assumption that $I$ is left-continuous is without loss of generality, since neither \eqref{TheoremM1} or \eqref{TheoremM2} is affected by substituting non-decreasing $I$ with its left-continuous representative.

\begin{proof}
It is trivial that \ref{Allf} implies \ref{Nif}, so to prove the equivalence we assume \ref{Nif} and proceed to find the constant $C$ for which \ref{Allf} holds. Fix $f \in X$ non-negative. We have by \eqref{BProp4}
\begin{align}
\int_t^{\infty} \frac{f(s)}{I(s)} \bigg( \int_t^s \frac{1}{I(r)} \: dr\bigg)^{m-1}  \:  ds &= (m-1)! H_I^mf(t) & \text{for} \ t \in (0, \infty). \label{ByH}
\end{align}
Because the function $H_I^mf$ is non increasing and $(0, \infty)$ is non-atomic, and therefore by Theorem~\ref{CharRes} resonant,  we get by Proposition~\ref{PropAss}
\begin{equation*}
\lVert H_I^mf \rVert_Y = \sup_{\lVert g \rVert_{Y'} \leq 1} \int_0^{\infty} g^*(t)H_I^mf(t) \: dt.
\end{equation*}
Remembering that $H_I^m$ and $R_I^m$ are mutually associate\footnote{See Proposition~\ref{BProp}}, we have
\begin{equation}
\lVert H_I^mf \rVert_Y = \sup_{\lVert g \rVert_{Y'} \leq 1} \int_0^{\infty} f(t) R_I^mg^*(t) \: dt. \label{NormHbyR}
\end{equation}
Thanks to rearrangement invariance of the norm $\lVert \cdot \rVert_X$  and \eqref{ByH}, our assumption~\ref{Nif} tells us
\begin{equation*}
C' = C' \sup_{\lVert f \rVert_{X} \leq 1} \lVert f \rVert_{X} = \sup_{\lVert f \rVert_{X} \leq 1} C' \lVert  f^* \rVert_{X} \geq  \sup_{\lVert f \rVert_{X} \leq 1} (m-1)!  \lVert H_I^mf^* \rVert_Y.
\end{equation*}
Hence, by applying \eqref{NormHbyR} (with $f$ replaced by $f^*$) and interchanging the suprema, we get
\begin{equation*}
\begin{split}
C' &\geq (m-1)! \sup_{\lVert f \rVert_{X} \leq 1} \sup_{\lVert g \rVert_{Y'} \leq 1} \int_0^{\infty} f^*(t) R_I^mg^*(t) \: dt  \\
 &= (m-1)! \sup_{\lVert g \rVert_{Y'} \leq 1} \sup_{\lVert f \rVert_{X} \leq 1} \int_0^{\infty} f^*(t) R_I^mg^*(t) \: dt  \\
 &= (m-1)! \sup_{\lVert g \rVert_{Y'} \leq 1} \lVert  R_I^mg^*(t) \rVert_{X'_d}
\end{split}
\end{equation*}
by the very definition of $\lVert \cdot \rVert_{X'_d}$ (Definition~\ref{DefDownAss}). Now, it follows from Theorem~\ref{TheoremE} that
\begin{equation*}
 \lVert  R_I^mg^*(t) \rVert_{X'} \leq 2^{m+1}  \lVert  R_I^mg^*(t) \rVert_{X'_d}
\end{equation*}
and therefore we have
\begin{equation*}
2^{m+1}C' \geq (m-1)! \sup_{\lVert g \rVert_{Y'} \leq 1} \lVert  R_I^mg^*(t) \rVert_{X'}
\end{equation*}
from which, by Defintion~\ref{DefAss}, Proposition~\ref{PropAss}, interchanging the suprema again and using the mutual associativity of $H_I^m$ and $R_I^m$, we conclude
\begin{equation*}
\begin{split}
2^{m+1}C' & \geq (m-1)! \sup_{\lVert g \rVert_{Y'} \leq 1} \lVert  R_I^mg^*(t) \rVert_{X'} \\
 & = (m-1)! \sup_{\lVert g \rVert_{Y'} \leq 1}  \sup_{\lVert f \rVert_{X} \leq 1} \int_0^{\infty} f(t) R_I^mg^*(t) \: dt  \\
 &= (m-1)! \sup_{\lVert f \rVert_{X} \leq 1} \sup_{\lVert g \rVert_{Y'} \leq 1} \int_0^{\infty} f(t) R_I^mg^*(t) \: dt \\
 &= (m-1)! \sup_{\lVert f \rVert_{X} \leq 1} \sup_{\lVert g \rVert_{Y'} \leq 1} \int_0^{\infty} g^*(t) H_I^mf(t) \: dt =  (m-1)! \sup_{\lVert f \rVert_{X} \leq 1} \lVert H_I^mf \rVert_Y.
\end{split}
\end{equation*}
That is, for every non-negative $f \in X$ the desired inequality holds for $C = 2^{m+1}C'$ and \ref{Allf} is therefore proved including the additional assertion.
\end{proof}

As a corollary, we get the Theorem~\ref{MainResultLite}.

\bibliographystyle{abbrv}
\renewcommand{\refname}{Bibliography}
\bibliography{bibliography}

\end{document}